\documentclass[12pt,reqno]{amsart}
\usepackage{amsmath, amssymb, amsfonts, xypic}
\numberwithin{equation}{section}
\newtheorem{thm}{Theorem}[section]
\newtheorem{prop}[thm]{Proposition}
\newtheorem{lem}[thm]{Lemma}
\newtheorem{dfn}[thm]{Definition}
\newtheorem{example}[thm]{Example}
\newtheorem{remark}[thm]{Remark}
\newtheorem{cor}[thm]{Corollary}
\newtheorem{note}[thm]{Note}

\usepackage[a4paper, top = 1.2in, bottom = 1.2in, left = 1.2in, right = 1.2in]{geometry}
\numberwithin{equation}{section}

\newcommand{\C}{\mathbb{C}}
\newcommand{\N}{\mathbb{N}}
\newcommand{\Q}{\mathbb{Q}}

\newcommand{\Z}{\mathbb{Z}}

\newcommand{\f}{\mathbf{f}}

\newcommand{\SL}{\mathrm{SL}}

\newcommand{\mcO}{\mathcal{O}}

\newcommand{\mfc}{\mathfrak{c}}

\newcommand{\mfH}{\mathfrak{H}}

\newcommand{\mfm}{\mathfrak{m}}

\newcommand{\mfp}{\mathfrak{p}}

\def\1{1\!\!1}

\newcommand{\psmat}[4]{\bigl( \begin{smallmatrix} #1 & #2 \\ #3 & #4 \end{smallmatrix} \bigr)}

\def\dis{\displaystyle}

\usepackage{tikz-cd}
\title{On non-vanishing of the Fourier coefficients of primitive forms}

\author[T. Dalal]{Tarun Dalal}
\address[T. Dalal]{Department of Mathematics, Indian Institute of Technology Hyderabad, Kandi, Sangareddy 502285, INDIA.}
\email{ma17resch11005@iith.ac.in}

\author[N. Kumar]{Narasimha Kumar}
\address[N. Kumar]{Department of Mathematics, Indian Institute of Technology Hyderabad, Kandi, Sangareddy 502285, INDIA.}
\email{narasimha.kumar@iith.ac.in}

\date{\today}

\begin{document}
\begin{abstract}
In this semi-expository article, we discuss about the non-vanishing of the Fourier coefficients of primitive forms.
We shall  make a note of a discrepancy in the statement of~\cite[Lemma 2.2]{KRW07}.
\end{abstract}
\maketitle


\section{Introduction}
\label{Introduction}
In $1947$, Lehmer conjectured that Ramanujan's tau function $\tau(n)$ is non-vanishing for all $n$. In~\cite{Leh47}, he
proved that the smallest $n$ for which $\tau(n)= 0$ must be a prime and showed that $\tau(n) \neq 0$ for all $n<33,16,799$.
It is well-known that the Fourier coefficients of Ramanujan's Delta function $\Delta(z)$ are in fact $\tau(n)(n \in \N)$. Note that
$\Delta(z)$ is a cuspidal Hecke eigenform of weight $12$ and level $1$.
It is a natural question to ask if a similar phenomenon continue to hold for cusp forms of higher weight and higher level.

In this semi-expository article, we study  the non-vanishing of the Fourier coefficients of primitive forms 
of any weight and any level. We take this opportunity to make a correction in the statement of~\cite[Lemma 2.2]{KRW07}.

\section{Preliminary}
In this section, we shall define modular forms and recall some basic facts about them. For more details, we refer the reader to
consult~\cite{DS05},~\cite{Miyake06}.

\subsection{Congruence subgroups}

The modular group $\SL_2(\Z)$ is defined by
$$\SL_2(\Z) := \left\{ \psmat{a}{b}{c}{d} : a,b,c,d \in \Z, ad- bc = 1 \right \}.$$ 
For any $N \in \N$, we shall define a subgroup of $\SL_2(\mathbb{Z})$ 
by 
$$\Gamma(N) = \{ \gamma \in \SL_2(\mathbb{Z}) \mid \gamma \equiv \psmat{1}{0}{0}{1} \pmod N\}.$$

\begin{dfn}
We say that a subgroup $\Gamma$ of $\SL_2(\mathbb{Z})$ is a congruence subgroup, if $\Gamma$ contains $\Gamma(N)$ 
for some $N\in \mathbb{N}$. 
\end{dfn}
In this theory, the following congruence subgroups play an important role
\begin{align*}
    \Gamma_1(N) &= \{ \gamma \in \SL_2(\mathbb{Z}) \mid \gamma \equiv \psmat{1}{*}{0}{1} \pmod N\}, \\ 
    \Gamma_0(N) &= \{ \gamma \in \SL_2(\mathbb{Z}) \mid \gamma \equiv \psmat{*}{*}{0}{*} \pmod N\} \ \ \ for\ any\ N \in \N.
\end{align*}
The subgroup $\Gamma(N)$ is called the principal congruence subgroup of $\SL_2(\mathbb{Z})$. 
Note that $\Gamma(N) \leq \Gamma_1(N)\leq \Gamma_0(N)\leq \SL_2(\mathbb{Z}),$ and
$\Gamma(1) = \Gamma_1(1)= \Gamma_0(1) = \SL_2(\mathbb{Z}).$

The modular group $\SL_2(\mathbb{Z})$ acts on the complex upper half plane 
$\mfH = \{ \tau \in \mathbb{C} \mid \mathrm{Im}(\tau) >0 \}$ via 
$$\gamma\tau = \frac{a\tau + b}{c\tau + d},$$ 
where $\tau \in \mfH, \ \gamma = \psmat{a}{b}{c}{d}\in \SL_2(\Z).$
For more details, please refer to~\cite[\S 1.2]{DS05}.
\subsection{Modular forms}
%
%
%

In this section, we shall define modular forms and recall some results related to them.

Let $X$ be the space of all complex valued holomorphic functions on $\mfH$. 
We can define an action of $\SL_2(\Z)$ on $X$ by using the action of $\SL_2(\Z)$ on $\mfH$
as follows.
For any $k\in \N$, $f\in X$ and $\gamma \in \SL_2(\mathbb{Z})$,
we define  the slash operator 
$$(f |_k \gamma )(\tau) := (c\tau + d)^{-k} f(\gamma\tau), \ \tau\in \mfH,$$
where $\gamma=\psmat{a}{b}{c}{d}$. Now, we define the notion of modular forms for any congruence subgroup $\Gamma$ of $\SL_2(\Z)$.
\begin{dfn}
Let $\Gamma$ be a congruence subgroup of $\SL_2(\mathbb{Z})$. A function $f \in X$ is said to be a \textbf{modular form} of weight $k$ with respect to $\Gamma$ if 
\begin{enumerate}
\item $f|_k \gamma = f , \forall \gamma \in \Gamma $, \ 
\item $f|_k\alpha$ is holomorphic at $\infty$, $\forall \alpha \in \SL_2(\mathbb{Z})$. 
\end{enumerate}
\end{dfn}
\begin{remark}
Note that one needs to verify condition $(2)$ only for the representatives of distinct cosets of $\Gamma$ in $\SL_2(\Z)$.  
\end{remark}
Now, we explain the meaning of $f$ being holomorphic at $\infty$. From condition $(1)$, it is clear that
then $f$ will be $h\Z$-periodic, where $h$ is the smallest integer such that $\psmat{1}{h}{0}{1}\in \Gamma$ (such $h$ exists since $\Gamma(N)\leq \Gamma$). 
This implies that there exists a function $g : D^{\prime} \longrightarrow \mathbb{C}$,  where $D^{\prime}$ is unit puncture disk, 
such that $f(\tau)= g(q_h)$ for all $\tau \in \mfH$, where $q_h=e^{\frac{2\pi i \tau}{h}}$.
It is clear that, the function $g$ is holomorphic on $D^{\prime}$, since $f$ is so on $\mfH$. 
The function $f$ is said to be \textbf{holomorphic at $\infty$} if $g$ extends holomorphically to $q=0$.
Similarly, one can define the meaning of $f|_k\alpha$ being holomorphic at $\infty.$ For more details, please refer to~\cite[\S 1.1, \S 1.2]{DS05}.

We denote the space of all modular forms of weight $k$ and level $\Gamma$ by $M_k(\Gamma)$.

\subsection{Fourier expansion}
Let $f \in M_k(\Gamma)$. Let $h$ be the smallest integer such that $\psmat{1}{h}{0}{1}\in \Gamma$.
Since $f$ is holomorphic at $\infty$, then $f$ has a Fourier expansion 
$$f(\tau) = \sum_{n=0}^\infty a_f(n)q_h^n, \ \ \mathrm{where} \ \ q_h= e^{\frac{2\pi i \tau}{h}} $$
for $\tau \in \mfH$.

\begin{dfn}
Let $f \in M_{k}(\Gamma)$. We say that $f$ is  a \textbf{cusp form} if  $a_{f|_k\alpha}(0)= 0$ for all $\alpha \in \SL_2(\mathbb{Z})$. 
We denote the space of all cusp forms of weight $k$ and level $\Gamma$ by $S_k(\Gamma)$. 
\end{dfn}
Note that $M_k(\Gamma), S_k(\Gamma)$ are vector spaces over $\C$. 
By~\cite[Theorem 3.5.1 and Theorem 3.6.1]{DS05}, these are in fact finite dimensional vector spaces over $\C$.
Now, we shall give some examples of modular forms and cusp forms.

\begin{example}
For any $k \geq 2$, we define the Eisenstein series of weight $2k$
 $$G_{2k}(\tau) = \sum_{(c,d) \in \mathbb{Z}^2 -\{(0,0)\}} \frac{1}{(c\tau + d)^{2k}}\in M_{2k}(\SL_2(\mathbb{Z})).$$
\end{example}
It is easy to check that $G_{2k}$ is a modular form of weight $2k$ and level $1$ (cf.~\cite[Page 4]{DS05}). 
The Fourier expansion of $G_{2k}$ at $\infty$ is given by 
\begin{equation}
 G_{2k}(\tau) = 2\zeta(2k) + 2 \frac{(2\pi i )^{2k}}{(2k-1)!}\sum_{n=1}^\infty \sigma_{2k-1}(n)q^n, \ \ k \geq 1,
\end{equation}
where $\sigma_{2k-1}(n)= \sum_{m\mid n, m>0} m^{2k-1}$. The normalized Eisenstein series is defined by 
$E_{2k}(\tau):= \frac{G_{2k}(\tau)}{2\zeta(2k)}$. Therefore, the Fourier expansion of $E_{2k}$ at $\infty$ is given by
$$E_{2k}(\tau)= 1- \frac{4k}{B_{2k}}\sum_{n=1}^\infty \sigma_{2k-1}(n)q^n, $$
where $B_k$'s are the Bernoulli numbers (cf.~\cite[Page 10]{DS05}).  

\begin{example}
From the dimensions of $S_k(\SL_2(\Z))$, one can see that $12$ is the least integer for which there is a non-zero cusp form for $\SL_2(\Z)$.
Moreover, dimension of $S_{12}(\SL_2(\Z))$ is $1$ and it is spanned by 
$$  \Delta(z) = (60G_4(z))^3 - 27(140G_6(z))^2 \in S_{12}(\SL_2(\mathbb{Z})), \ \ z\in \mfH.$$
\end{example}
The product formula for $\Delta(z)$ is given by
$\dis \Delta(z)= q\prod_{n\geq 1}(1-q^n)^{24}=\sum_{n \geq 1} \tau(n)q^n,$
where $q=e^{2\pi i z}$.

\begin{example}[\cite{Shi71}, Example 2.28]
 For $N\in \{2,3,5,11 \}$, $(\Delta(z)/\Delta(Nz))^{1/(N+1)}\in S_{24/(N+1)}(\Gamma_0(N))$. 
 Moreover, the space $S_{24/(N+1)}(\Gamma_0(N))$ is one dimensional and it is spanned by $(\Delta(z)/\Delta(Nz))^{1/(N+1)}$.
 
\end{example}

\subsection{Modular forms with character}

A Dirichlet character modulo $N$ is a group homomorphism $\chi : (\mathbb{Z}/N\mathbb{Z})^* \longrightarrow \mathbb{C}^*$.

\begin{dfn}
The space of all modular forms of weight $k$ level $N$ with character $\chi$ is defined by 
$$M_k(N,\chi) = \{ f \in M_k(\Gamma_1(N)) | f|_k\psmat{a}{b}{c}{d} = \chi(d)f,\forall \psmat{a}{b}{c}{d} \in \Gamma_0(N)\}.$$
\end{dfn}
The space $M_k(\Gamma_1(N))$ decomposes as  $$M_k(\Gamma_1(N))= \bigoplus_{\chi}M_k(N,\chi), $$
where $\chi$ varies over all Dirichlet characters of $(\Z/N\Z)^*$ such that $\chi(-1)= (-1)^k$ (cf. \cite[Lemma 4.3.1]{Miyake06}).
Similarly one can define the space of cusp forms of weight $k$ level $N$ with character $\chi$ and they are denoted by $S_k(N,\chi)$.
One can easily check that 
 $S_k(N,\chi)  =  S_k(\Gamma_1(N)) \cap M_k(N,\chi) .$
Moreover, a similar decomposition holds as well, i.e.,  
$$S_k(\Gamma_1(N))= \bigoplus_{\chi} S_k(N,\chi), $$
where $\chi$ varies over all Dirichlet characters of $(\Z/N\Z)^*$ with $\chi(-1)= (-1)^k$ (cf. \cite[Lemma 4.3.1]{Miyake06}).

\begin{example}[Poincar\'e series]
 Let $\Gamma_{\infty}=\{ \psmat{1}{b}{0}{1} \mid b\in \Z \}$, and $\chi$ be any Dirichlet character modulo $N$.
 For $m\geq 1$, we define 
 \begin{equation*}
   P_m(z) := \sum_{\gamma=\psmat{a}{b×}{c×}{d×}\in \Gamma_{\infty} \char`\\ \Gamma_0(N)} \overline{\chi}(\gamma)\frac{1}{(c z+d)^k}\mathrm{exp}(2\pi i m\gamma z).
 \end{equation*}
for any integer $k\geq 2$. By~\cite[Proposition 14.1]{IK04}, $P_m(z) \in S_k(N,\chi).$ 
\end{example}

Now we will define two types of operators on the space of modular forms (resp., cusp forms).
They are known as Hecke operators.
\subsection{Hecke operators}
Let $M_k(\Gamma_1(N))$ be a space of modular forms of weight $k$, level $N$. For any $(n,N)=1$, we define 
the \textbf{diamond operator} $$\langle n \rangle:M_k(\Gamma_1(N)) \longrightarrow M_k(\Gamma_1(N))$$ as
$$\langle n \rangle f := f|_k\alpha, \mathrm \ {for\ any} \ \alpha=\psmat{a}{b}{c}{\delta}\in \Gamma_0(N) \ \mathrm{with} \ \delta\equiv n\pmod N. $$
 We can also extend the definition of diamond operator to $\N$ via $\langle n \rangle = 0$ if $(n,N)>1$.
Observe that for any character $\chi : (\mathbb{Z}/N\mathbb{Z})^* \longrightarrow \mathbb{C}^*,$
$$M_k(N,\chi) = \{ f \in M_k(\Gamma_1(N)) | \langle n \rangle f = \chi(n)f,\forall n\in (\mathbb{Z}/N\mathbb{Z})^*\}.$$
Note that, the diamond operator acts trivially on $M_k(\Gamma_0(N))$, since $M_k(\Gamma_0(N))=M_k(N,\chi^{\circ}_N)$,
where $\chi^\circ_N$ is the trivial character modulo $N$.

Now, we will define the second type of \textbf{Hecke operator} for any prime $p$, 
and they are denoted by $T_p.$
If $f(\tau) =   \sum_{n=0}^\infty a_f(n)q^n \in M_k(\Gamma_1(N))$, then 
$$  (T_pf)(\tau) = \sum_{n=0}^\infty a_f(np)q^n + \chi^\circ_N(p)p^{k-1}\sum_{n=0}^\infty a_{\langle p \rangle f}(n)q^{np} \in M_k(\Gamma_1(N)). $$
 Similarly, one can also defined the action of $T_p$ on $M_k(N,\chi)$ as follows:
If $f(\tau) =   \sum_{n=0}^\infty a_f(n)q^n \in M_k(N, \chi)$, then 
$$  (T_pf)(\tau) = \sum_{n=0}^\infty a_f(np)q^n + \chi(p)p^{k-1}\sum_{n=0}^\infty a_f(n)q^{np} \in M_k(N, \chi),$$
In fact, for $n\in \N$, one can define the Hecke operators $T_n$ as follows: 
\begin{enumerate}
  \item For any prime $p$ and $r\geq 2$ we define $T_{p^r} = T_pT_{p^{r-1}} - p^{k-1}\chi(p)T_{p^{r-2}}.$
   \item  For $n=p_1^{e_1}\ldots p_k^{e_k}$ we define $T_n=T_{p_1^{e_1}}\ldots T_{p_k^{e_k}}. $
\end{enumerate}
One can check that, any two primes $p \neq q$, $T_pT_q=T_qT_p$.
In fact, the Hecke operators respects the spaces $S_k(N,\chi)$ and $S_k(\Gamma_0(N))$. For more details, we refer the reader to 
~\cite[\S 5.3]{DS05}.

\subsection{Petersson inner product}
To study the space of cusp forms $S_k(\Gamma_1(N))$ further, we make it into an inner product  space.
In order to do so, we need to define an inner product on the space of cusp forms.

The \textbf{hyperbolic measure} on the upper half plane is defined by $$d\mu(\tau):= \frac{dxdy}{y^2}, \ \tau = x+iy \in \mfH.$$
For any congruence subgroup $\Gamma \leq \SL_2(\mathbb{Z})$, 
the \textbf{Petersson inner product} 
$$\langle ,\rangle _\Gamma:S_k(\Gamma)\times S_k(\Gamma)\longrightarrow \mathbb{C}$$ 
is given by $$ \langle f,g \rangle_\Gamma=\frac{1}{V_\Gamma}\int_{\Gamma \char`\\ \mfH}f(\tau)\overline{g(\tau)}(\mathrm{Im}(\tau))^kd\mu(\tau), \ \mathrm{where} \ V_\Gamma = \int_{\Gamma \char`\\ \mfH}d\mu(\tau).$$
This inner product is linear in $f$, conjugate linear in $g$, Hermitian symmetric and positive definite.
By~\cite[Theorem 5.5.3]{DS05}, the Hecke operators $\langle n \rangle$ and $T_n$ are normal operators for $(n,N)=1$.
By~\cite[Theorem 5.5.4]{DS05}, we have that
\begin{thm}
The space $S_k(\Gamma_1(N))$ has an orthogonal basis of simultaneous eigenforms for the Hecke operators $\{\langle n \rangle, T_n: (n,N)=1 \}.$
\end{thm}

Now, we shall introduce the theory of old forms and new forms. This in fact leads to define the notion of primitive forms.
(cf.~\cite[\S 5.4]{DS05} for more discussion on this).

\subsection{Old forms and New forms}
For $d|N$, we define the mapping $$i_d: (S_k(\Gamma_1(Nd^{-1})))^2 \longrightarrow S_k(\Gamma_1(N)) \ by\ $$
$$(f,g) \longrightarrow f + g|_k\psmat{d}{0}{0}{1}.$$
The space of \textbf{old forms} is defined by $$  S_k(\Gamma_1(N))^{\mathrm{old}} = \sum_{p|N} i_p((S_k(\Gamma_1(Np^{-1})))^2).$$
The space of \textbf{new forms} (denote by $S_k(\Gamma_1(N))^{\mathrm{new}}$) is defined to be the orthogonal complement of $S_k(\Gamma_1(N))^{\mathrm{old}}$ with respect to the Petersson inner product.
By~\cite[Proposition 5.6.2]{DS05}, we see that the spaces $S_k(\Gamma_1(N))^{\mathrm{old}}$ and 
$S_k(\Gamma_1(N))^{\mathrm{new}}$ are stable under the action of $T_n$
and $\langle n \rangle$ for all $n \in \N$.
\begin{dfn}
A primitive form is a normalized eigenform in $f \in S_k(\Gamma_1(N))^{\mathrm{new}}$, i.e.,
$f$ is an eigenform for the Hecke operators $T_n,\langle n \rangle$ for all $n\in \mathbb{N}$, and
$a_f(1)= 1$.
\end{dfn}
By~\cite[Theorem 5.8.2]{DS05}, the set of primitive forms in the space $S_k(\Gamma_1(N))^{\mathrm{new}}$ forms an orthogonal basis. 
Each such primitive form lies in an eigen space $S_k(N,\chi)$ for an unique character $\chi$. In fact, its Fourier coefficients are its $T_n$-eigenvalues.
\begin{note}
When we say that $f \in S_{k}(N,\chi)$ is a primitive form of weight $k$, level $N$, with character $\chi$, 
actually we mean $f\in S_k(\Gamma_1(N))^{\mathrm{new}}$ is a primitive form and it belongs the eigenspace $S_k(N,\chi)$.
\end{note}


\begin{prop}[\cite{DS05}, Proposition 5.8.5]
\label{properties}
Let $f = \sum_{n=1}^\infty a_f(n)q^n \in S_k(N,\chi)$. Then $f$ is a normalized eigenform 
if and only if 
its Fourier coefficients satisfy the following relations
\begin{enumerate}
\item $a_f(1)=1$,
\item $a_f(m) a_f(n) = a_f(m)a_f(n)$ if $(m,n)=1$,
\item $a_f(p^r) = a_f(p) a_f(p^{r-1}) - p^{k-1}\chi(p)a_f(p^{r-2})$, for all prime $p$ and $r\geq 2$.
\end{enumerate}
\end{prop}
For more details on this content, please refer to~\cite[\S 5.7, \S 5.8]{DS05}.
\section{Classical modular forms}
Recall that, Lehmer proved that the smallest $n$ for which $\tau(n)= 0$ must be a prime. 
We are interested in studying a similar question for the Fourier coefficients of primitive forms of higher weight and higher level. 
Let $f = \sum_{n=1}^{\infty} a_f(n)q^n \in S_{k}(N,\chi)$ be a primitive form of even weight $k$, level $N$, with character $\chi$.	

Suppose that $a_f(n)=0$ for some $n = \prod_i p_i^{r_i} \geq 1$.
By Proposition~\ref{properties}, we see that $a_f(p_i^r)=0$
for some prime $p_i$. 
In this section, we shall explore the relation between the vanishing (resp., non-vanishing) of $a_f(p)$
and $a_f(p^r)$ for $r\geq 2$.
We begin this discussion with a lemma of Kowalski, Robert, and Wu (see \cite[Lemma 2.2]{KRW07}).
\begin{prop}
\label{lemma2.2}
Let $f = \sum_{n=1}^{\infty} a_f(n)q^n \in S_{k}(N,\chi)$ be a primitive form of even weight $k$, level $N$, with character $\chi$.
There exists an integer $M_f \geq 1$,  such that for any prime $p \not \mid M_f$, either $a_f(p)=0$ or $a_f(p^r)\neq 0$ for all $r\geq 1$. 
\end{prop}

\begin{proof}
If $p \mid N$ then $a_f(p^r) = a_f(p)^r$ for any $r \geq 1$, so in this case  the conclusion holds trivially.
Let $p$ be a prime number such that $p \not \mid N$. 
If $a_f(p) = 0$, then there is nothing prove. 
Suppose that $a_f(p) \neq 0$ but $a_f(p^r) = 0$ for some $r\geq 2$.
Since $f$ is a primitive form, then by Hecke relations, we have
\begin{equation*}
    a_f(p^{m+1}) = a_f(p)a_f(p^m) - \chi(p)p^{k-1} a_f(p^{m-1})
\end{equation*}
for any $m \in  \N$.
These relations can be re-interpreted as
\begin{equation}
\sum_{r=0}^\infty a_f(p^r)X^r = \frac{1}{1-a_f(p)X+\chi(p)
p^{k-1}X^2}.
\end{equation}
Suppose that
\begin{equation}\label{Quad}
    1-a_f(p)X+\chi(p)
p^{k-1}X^2 = (1- \alpha(p)X)(1-\beta(p)X).
\end{equation}
By comparing the coefficients, we get that
\begin{equation*}
    \alpha(p) + \beta(p) = a_f(p) \ \ \ \ \ \ \mathrm{and} \ \ \ \ \ \ \alpha(p)\beta(p) = \chi(p)p^{k-1} \neq 0,
    \end{equation*}
since $p \not \mid N$ and hence $\chi(p) \neq 0$. If $\alpha(p) = \beta(p)$, then
\begin{equation}
\label{equalcannot}
    a_f(p^t) = (t+1)\alpha(p)^t\neq 0,
\end{equation}
for any $t \geq 2$ and this cannot happen. Therefore, $\alpha(p) \neq  \beta(p)$.
Then, by induction, we have the following
\begin{equation*}
    a_f(p^t) = \frac{\alpha(p)^{t+1} - \beta(p)^{t+1}}{\alpha(p)- \beta(p)}.
\end{equation*}
for any $t \geq 2$.  Recall that $a_f(p^r)=0$ for some $r \geq 2$. Therefore,
\begin{equation}
\label{rootofunity}
    a_f(p^r) = 0 \ \ \mathrm{if\ and\ only\ if} \ \ \Bigg(\frac{\alpha(p)}{\beta(p)}\Bigg)^{r+1} = 1,
\end{equation}
which implies that the ratio $\frac{\alpha(p)}{\beta(p)}$ is a  ($r+1$)-th root of unity. Since $a_f(p) \neq 0$, 
we get that $\alpha(p) = \zeta \beta(p)$ where $\zeta$ is a root of unity and $\zeta \ne -1$ . By the product relation, we get that
$\alpha(p)^2= \zeta\chi(p)p^{k-1}$, hence $\alpha(p) = \pm \gamma p^{(k-1)/2} $, 
where $\gamma^2 = \zeta \chi(p)$. Therefore,
\begin{equation*}
    a_f(p) = (1+\zeta^{-1})\alpha(p) = \pm \gamma (1+\zeta^{-1})p^{(k-1)/2} \ne 0.
\end{equation*}
In particular, $\gamma (1+\zeta^{-1})p^{(k-1)/2} \in  \mathbb{Q}(f)$,
where $\Q(f)$ is the number field generated by the Fourier coefficients of $f$ and by the values of  $\chi$.
Since $k$ is even, we have  
\begin{equation}
\label{main equation}
    \gamma(1+\zeta^{-1})\sqrt{p} \in \mathbb{Q}(f).
\end{equation}
We have that the number of such primes $p$ are finite, since $\mathbb{Q}(f)$ is a number field.
Take $M_f$ to be  the product of all such primes $p$. Thus, for any prime $p\not \mid M_f$, we have
either $a_f(p) = 0$ or $a_f(p^r)\neq 0$ for all $r\geq 1$. \end{proof}

\begin{cor}
Let $f, M_f$ be as in the above Proposition. Then the smallest $m \in \N$ with $(m,M_f)=1$ with $a_f(m)=0$ is a prime.
\end{cor}
If $M_f=1$, then the corollary is exactly the generalization of Lehmer's result that
that the smallest $n$ for which $\tau(n)= 0$ must be a prime. Now, this leads to the question of
calculating $M_f$ for $f$. In the second part of~\cite[Lemma 2.2]{KRW07}, it was stated as follows:
\begin{prop}
\label{wrong}
Let $f,M_f$ be as in Proposition~\ref{lemma2.2}. If the character $f$ is trivial and the Fourier coefficients of $f$ are integers,
then one can take $M_f=N$. 
\end{prop}

However, we are able to produce examples which contradicts this statement. 

\begin{example}
Let $E$ be an elliptic curve defined by the minimal Weierstrass equation $y^2 + y = x^{3}- x$. The Cremona label for $E$ is  $37a1$.
Let $f_E$ denote the primitive form (of weight $2$ and level $37$) associated to $E$ by the modularity theorem.
The Fourier expansion of $f_E$ is given by 
$$f_E(q)= \sum_{n=1}^{\infty}a_{f_E}(n)q^n= q-2q^2-3q^3+2q^4-2q^5+6q^6-q^7+6q^9+O(q^{10}).$$ 
Note that $(2,37)=1$ and $a_{f_E}(2)$ is non-zero but $a_{f_E}(8)=0$.
\end{example}

\begin{example}
Let $E$ be an elliptic curve defined by the minimal Weierstrass equation $y^2+xy+y= x^3-x^2$. The Cremona label for $E$ is $53a1$.
Let $f_E$ denote the primitive form (of weight $2$ and level $53$) associated to $E$ by the modularity theorem.
The Fourier expansion of $f_E$ is given by
$$f_E(q)= \sum_{n=1}^{\infty}a_{f_E}(n)q^n=q-q^2-3q^3-q^4+3q^6-4q^7+3q^8+6q^9+O(q^{10}).$$ 
Note that $(3,53)=1$ and $a_{f_E}(3)$ is non-zero but a simple calculation using the relations among the Fourier coefficients shows that $a_{f_E}(3^5)=0.$
\end{example}
For the convenience of the reader, we shall recall their proof of Proposition~\ref{wrong}.

\begin{proof}
Let $p$ be a prime number such that $p \not \mid N$.  If $a_f(p) = 0$, then there is nothing prove. 
Suppose $a_f(p) \neq 0$ but $a_f(p^r) = 0$ for some $r\geq 2$.
Arguing as in Proposition~\ref{lemma2.2}, the argument is valid till~\eqref{main equation}.
After that, they wished to show that \eqref{main equation} does not hold for any prime $p \not \mid N$.

By~\eqref{Quad},~\eqref{rootofunity}, we get that $\frac{\alpha(p)}{\beta(p)}= \zeta$ is a root of unity in a quadratic extension of $\Q$,
hence $\zeta \in \{-1, \pm i, \pm \omega_3 , \pm \omega_3^2\}$. All those except $\zeta = - 1$ contradict the fact that $f$ has integer coefficients by simple considerations
       such as the following, for $\zeta = \omega_3$ say: we have $\alpha(p)^2 = \omega_3 p^{k-1}$, $\gamma = \pm \omega_3^2 p^\frac{k-1}{2}$
       and $\lambda_f(p) = (1 + \omega^{-1}_3) \gamma = \pm (1 + \omega^{-1}_3)\omega_3^2 p^{\frac{k-1}{2}} = \pm (\omega_3^2 + \omega_3) p^{\frac{k-1}{2}} \not \in \Z$.
       Therefore, \eqref{main equation} does not hold for any prime $p \not \mid N$.
\end{proof}

In the last part of the above proof, when we calculated the expression in $(\ref{main equation})$ for $\zeta \neq \pm 1$,
it seem to hold for $p=2$ (resp., $p=3$) with some special values of $a_f(2)$ (resp., $a_f(3)$).
In the next proposition, we have calculated the optimal value of $M_f$ and the correct version of Proposition~\ref{wrong} 
is
\begin{prop}
\label{Our}
Let $f, M_f$ be  as in Proposition~\ref{lemma2.2}. If the character $\chi$ is trivial and the Fourier coefficients of $f$ are integers, then $M_f$ can be so chosen that $(M_f,N)=1$ and $M_f \mid 6$.
\end{prop}
\begin{proof}
If $p \mid N$ then $a_f(p^r) = a_f(p)^r$ for any $r \geq 1$, so in this case  the conclusion of Proposition~\ref{lemma2.2} holds trivially. Hence, the number $M_f$ is relatively prime to $N$.

If $p \not \mid N$, we argue as in the proof of Proposition~\ref{wrong} till the last step. Now, we compute~\eqref{main equation} 
for all values of $\zeta$ to prove our proposition. Let $\omega_n$ denote $e^{\frac{2\pi i}{n}}$ for any $n \in \N$.
 
\begin{enumerate}
 \item The root of unity $\zeta$ cannot be $1$ because of~$\eqref{equalcannot}$.
 \item The root of unity $\zeta$ cannot be $-1$ because $0 \neq a_f(p) = \alpha(p)+\beta(p)$.
 \item  If $\zeta = \omega_3$, then $\alpha(p)^2= \omega_3 p^{k-1} \Rightarrow \alpha(p)= \pm \omega_3^2 p^{\frac{k-1}{2}}$. 
        This implies that $a_f(p) = \pm(1+ \omega_3^2) \omega_3^2p^{\frac{k-1}{2}}= \mp p^{\frac{k-1}{2}} \not \in \mathbb{Z}$. 
        For $\zeta = \omega_3^2$, we will get the same conclusion.

 \item If $\zeta = i$, then $\alpha(p)^2= ip^{k-1} \Rightarrow \alpha(p)= \pm \omega_8 p^{\frac{k-1}{2}} \Rightarrow a_f(p)= \pm (1-i) \omega_8 p^{\frac{k-1}{2}}= \pm \sqrt{2}p^{\frac{k-1}{2}}$. 
       This implies that 
       $$\sqrt{2}p^{\frac{k-1}{2}}\in \mathbb{Z} \iff p= 2,$$
       in which case $a_f(2)= \pm 2^{k/2}$. For $\zeta =-i$, we will get the same conclusion.
 
 \item If $\zeta = -\omega_3$, then $\alpha(p)^2= -\omega_3 p^{k-1} \Rightarrow \alpha(p)=\pm\frac{\sqrt{3}-i}{2}p^{\frac{k-1}{2}} 
       \Rightarrow a_f(p)=\pm(1+\frac{1+i\sqrt{3}}{2})\frac{\sqrt{3}-i}{2}p^{\frac{k-1}{2}}=\pm\sqrt{3}p^{\frac{k-1}{2}}$.
       This implies that
       $$\sqrt{3}p^{\frac{k-1}{2}}\in \mathbb{Z} \iff p=3,$$ 
       in which case $a_f(3)= \pm 3^{k/2}$. If $\zeta = -\omega_3^2$, then we will get same conclusion.
 \end{enumerate}
This case by case analysis would imply that $M_f$ is a divisor of $6$. This means that the possible values of $M_f$ are $1,2,3,6$.
\end{proof}
%
%

For any prime $p$, $\chi^\circ_p$ denote the trivial character on $\left(\Z/p\Z\right)^*$, i.e.,
for any $N \in \N$, we have
$$\chi^\circ_p (N) := 
                \begin{cases}
                0 & \text{if}\quad  p \mid N,\\
                1 & \text{if}\quad  p \not \mid N.
                \end{cases}  $$
Based on the proof of the above proposition, we can re-interpret the above result as follows:
\begin{lem}
\label{OurCorollary}
Let $f, M_f$ be as in Proposition~\ref{Our}. Then $M_f$ can be taken to be $2^{\chi^\circ_2(N)}3^{\chi^\circ_3(N)}$. 
Further if
\begin{itemize}
 \item $2 \mid M_f$, $a_f(2) \neq \pm 2^{k/2}$, then $2$ can be dropped  from $M_f$, i.e., $M_f$ can  be taken to be $3^{\chi^\circ_3(N)}$,
 \item $3 \mid M_f$, $a_f(3) \neq \pm 3^{k/2}$, then $3$ can be dropped  from $M_f$, i.e.,  $M_f$ 
  can  be taken to be  $2^{\chi^\circ_3(N)}$,
 \item $6 \mid M_f$, $a_f(p) \neq \pm p^{k/2}$(for $p=2,3$), then $6$ can be dropped  from $M_f$, i.e., $M_f$  can  be taken to be  $1$.
\end{itemize}
\end{lem}
Note that the above lemma gives an optimal $M_f$ for which Proposition~\ref{Our}
continues to hold. The following corollaries describes the nature of the first vanishing of Fourier coefficients of primitive forms of higher weight $k$ and higher level $N$.

\begin{cor}\label{Relatively prime}
Let $f=\sum_{n=1}^{\infty} a_f(n) q^n \in S_{k}(\Gamma_0(N))$ be a primitive form of even weight $k$ and level $N$
with $a_f(n)\in \mathbb{Z}$. Let $M_f$ be as in Lemma~\ref{OurCorollary}.
Then the smallest $n \in \N$ with $(n,M_f)=1$ with $a_f(n)=0$ is prime.
\end{cor}
\begin{proof}
Let $n$ be the smallest integer with $(n,M_f)=1$ such that $a_f(n)=0$. 
Since $f$ is a primitive form, we know that the Fourier coefficients of $f$ satisfy
\begin{equation}\label{Fcoeff}
a_{f}(n_1n_2)= a_{f}(n_1)a_{f}(n_2) \ \ \mathrm{if}\ (n_1,n_2)=1.
\end{equation}
This forces that $n=p^r$, where $p$ is a prime with $(p,M_f)=1$.
By Proposition~\ref{Our}, we get that $r=1$.
Therefore $n$  has to be a prime.
\end{proof}
The following two corollaries can be thought of as a generalization of the result of Lehmer 
which states that the smallest $n$ for which $\tau(n)= 0$ must be a prime. 
\begin{cor}
Let $f = \sum_{n=0}^{\infty} a_f(n) q^n \in S_{k} (\Gamma_0(N))$ be a primitive form of even weight $k$, level $N$ with $a_f(n)\in \Z$. If $6$ divides $N$, then the smallest $n$ for which $a_f(n) = 0$ is  a prime. 
\end{cor}
\begin{proof}
Since $M_f \mid 6$, and $6 \mid N$, we have that $M_f \mid N$.
Since $(M_f,N)=1$, we have that $M_f=1$. By Corollary~\ref{Relatively prime},
the result follows.
\end{proof}
In order to get a similar conclusion as above for cusp forms when  $6 \not \mid N$, e.g., for $\Delta$-function, we need to impose some conditions on $a_f(2), a_f(3)$, which is the content of the following Corollary.
It follows from Lemma~\ref{OurCorollary} and  coincides with~\cite[Proposition 4.2]{TQ18},

\begin{cor}
Let $f = \sum_{n=0}^{\infty} a_f(n) q^n \in S_{k} (\Gamma_0(N))$ be a primitive form of even weight $k$, level $N$
with $a_f(n)\in \Z$. 
Suppose $a_f(2) \neq \pm 2^{\frac{k}{2}}$ and $a_f(3) \neq \pm 3^{\frac{k}{2}}$. 
Then the smallest $n$ for which $a_f(n) = 0$ is  a prime. 
\end{cor}
\begin{proof}
We know that $M_f \mid 6$ and $(M_f,N)=1$.  By Lemma~\ref{OurCorollary}, it follows that $M_f$ can be improved to $1$.
Therefore, the result follows by Corollary~\ref{Relatively prime}.
\end{proof}

\section{Hilbert modular forms}
There is a generalization of Proposition~\ref{lemma2.2} available in the context of Hilbert modular forms.
In fact, we used this generalization to study the simultaneous non-vanishing of Fourier coefficients of distinct primitive forms 
at powers of prime ideals (cf. ~\cite{DK19}). We shall state that generalization in this section.

Let $K$ be a totally real number field of odd degree $n$ and $\mathbb{P}$ denote the set of all prime ideals of $\mcO_K$ with odd inertia degree.
Let $\mathbf{P}$ denote the set of all prime ideals of $\mcO_K$.

Let $\f$ be a primitive form over $K$ of level $\mfc$, with character $\chi$ and weight $\mathbf{2k}=(2k_1,\ldots, 2k_n)$. Let $2k_0$ denote the maximum
of $\{2k_1, \ldots, 2k_n \}$. For each integral ideal $\mfm \subseteq \mcO_K$, let $C(\mfm,\f)$ denote the Fourier coefficients of $\f$ at $\mfm$.
 

Now, we state the result which is analogous to Proposition~\ref{lemma2.2} for $\f$.
\begin{prop}
\label{HilbertProposition}
Let $\f$ be a primitive form over $K$ of level $\mfc$, with character $\chi$ and
weight $\mathbf{2k}$. Then there exists an integer $M_\f \geq 1$ with $N(\mfc)\mid M_\f$ such that for any prime $p \not \mid M_\f$ and 
   for any prime ideal $\mfp \in \mathbb{P}$ over $p$,
   we have either $C(\mfp,\f) = 0$ or $C(\mfp^r,\f) \neq 0$ for all $r \geq 1$. 
\end{prop}
\begin{proof}
Let $p$ be a prime number such that $p \not \mid N(\mfc)$. Let $\mfp \in \mathbb{P}$ be a prime ideal of $\mcO_K$ over $p$ and $\mfp \not \mid \mfc$. 
If $C(\mfp,\f)=0$, then there is nothing prove. If $C(\mfp, \f) \neq 0$, then we need to show that $C(\mfp^r,\f) \neq 0$ for all $r\geq 2$,
except for finitely many prime ideals $\mfp \in \mathbb{P}$.  

Suppose that $C(\mfp, \f) \neq 0$ but $C(\mfp^r,\f) = 0$ for some $r\geq 2$.
Since $\f$ is a primitive form, then by Hecke relations, we have
\begin{equation*}
    C(\mfp^{m+1}, \f) = C(\mfp,\f)C(\mfp^m,\f) - \chi(\mfp)N(\mfp)^{2k_0-1} C(\mfp^{m-1},\f).
\end{equation*}
These relations can be re-interpreted as
\begin{equation}
\sum_{r=0}^\infty C(\mfp^r,\f)X^r = \frac{1}{1-C(\mfp,\f)X+\chi(\mfp)
N(\mfp)^{2k_0-1}X^2}.
\end{equation}
Suppose that
\begin{equation*}
    1-C(\mfp,\f)X+\chi(\mfp)
N(\mfp)^{2k_0-1}X^2 = (1- \alpha(\mfp)X)(1-\beta(\mfp)X).
\end{equation*}
By comparing the coefficients, we get that
\begin{equation*}
    \alpha(\mfp) + \beta(\mfp) = C(\mfp, \f) \ \ \ \ \ \ \mathrm{and} \ \ \ \ \ \ \alpha(\mfp)\beta(\mfp) = \chi(\mfp)N(\mfp)^{2k_0-1} \neq 0,
    \end{equation*}
since $\mfp \not \mid \mfc$ and hence $\chi(\mfp) \neq 0$. If $\alpha(\mfp) = \beta(\mfp)$, then 
\begin{equation*}
    C(\mfp^r,\f) = (r+1)\alpha(\mfp)^r\neq 0,
\end{equation*}
which cannot happen for any $r \geq 2$. So, $\alpha(\mfp)$ cannot be equal to $\beta(\mfp)$.
Then by induction, for any $r \geq 2$, we have the following
\begin{equation*}
    C(\mfp^r,\f) = \frac{\alpha(\mfp)^{r+1} - \beta(\mfp)^{r+1}}{\alpha(\mfp)- \beta(\mfp)}.
\end{equation*}
In this case, we have
\begin{equation*}
    C(\mfp^r,\f) = 0 \ \ \mathrm{if\ and\ only\ if} \ \ \Bigg(\frac{\alpha(\mfp)}{\beta(\mfp)}\Bigg)^{r+1} = 1,
\end{equation*}
which implies that the ratio $\frac{\alpha(\mfp)}{\beta(\mfp)}$ is a root of unity. Since $C(\mfp,\f) \neq 0$, 
we get that $\alpha(\mfp) = \zeta \beta(\mfp)$ where $\zeta$ is a root of unity and $\zeta \ne -1$ . By the product relation, we get that
$\alpha(\mfp)^2= \zeta\chi(\mfp){{N(\mfp)}^{2k_0-1}}$, hence $\alpha(\mfp) = \pm \gamma {N(\mfp)}^{{(2k_0-1)}/2} $, 
where $\gamma^2 = \zeta \chi(\mfp)$. Therefore,
\begin{equation*}
    C(\mfp,\f) = (1+\zeta^{-1})\alpha(\mfp) = \pm \gamma (1+\zeta^{-1}){N(\mfp)^{(2k_0-1)/2}} \ne 0.
\end{equation*}
In particular, $\mathbb{Q}(\gamma (1+\zeta^{-1})N(\mfp)^\frac{2k_0-1}{2}) \subseteq \mathbb{Q}(\f)$,
where $\Q(\f)$ is the field generated by $\{C(\mfm,\f)\}_{\mfm \subseteq \mcO_K}$ and by the values of the character $\chi$.
Since $\mfp \in \mathbb{P}$, $N(\mfp)= p^{f}$, where $f \in \N$ odd.
Hence, we have
\begin{equation}
    \mathbb{Q}(\gamma(1+\zeta^{-1})p^{\frac{f(2k_0-1)}{2}}) \subseteq \mathbb{Q}(\f).
\end{equation}
Since $2k_0-1$, $f$ are odd, we have that 
\begin{equation}
    \mathbb{Q}(\gamma(1+\zeta^{-1})\sqrt{p}) \subseteq \mathbb{Q}(\f).
\end{equation}
By~\cite[Proposition 2.8]{Shi78}, the field $\mathbb{Q}(\f)$ is a number field. Hence, the number of such primes $p$ are finite. Take $M_\f$ to be  the product of all such primes $p$ and $N(\mfc)$.
Thus, for any prime $p\not \mid M_\f$ and for any prime ideal $\mfp \in \mathbb{P}$ over $p$, we have
either $C(\mfp,\f) = 0$ or $C(\mfp^r,\f)\neq 0$ for all $r\geq 1$.
\end{proof}
We end this article with the following statement:
\begin{lem}
Let $\f$ and $K$ be as in Proposition~\ref{HilbertProposition}.
Further, if $K$ is Galois over $\Q$, then
there exists an integer $M_\f \geq 1$ with $N(\mfc)\mid M_\f$ such that for any prime $p \not \mid M_\f$ and 
   for any prime ideal $\mfp \in \mathbf{P}$ over $p$, 
   we have either $C(\mfp,\f) = 0$ or $C(\mfp^r,\f) \neq 0$ for all $r \geq 1$. 
\end{lem}
We note that  in a recent work of Bhand, Gun and Rath (cf.~\cite[Theorem 2]{BGR19}),
they have computed the lower bounds of the Weil heights of $C(\mfp^r,f)$, when non-zero,
for prime ideals $\mfp$ away from an ideal $\mathbf{M}$. In particular, 
the above lemma is a consequence of their Theorem.

\section*{Acknowledgements}  
 The authors are thankful to the anonymous referee for the valuable suggestions towards the improvement of this paper. 
 The first author thanks University Grants Commission (UGC), India for the financial support provided in the form of 
 Research Fellowship to carry out this research work at IIT Hyderabad. The second author's research was partially supported by the SERB grant MTR/2019/000137.

\end{document}